\documentclass[11pt]{amsart}
\usepackage{amsmath}
\usepackage{amssymb}
\usepackage{tabularx}
\usepackage{enumerate}
\usepackage{graphicx}
\usepackage{texdraw}

\usepackage{mathrsfs}

\usepackage{amsfonts,amssymb,amsmath}
\usepackage{epsfig}

\makeatletter
\@addtoreset{equation}{section}

\makeatother
\newtheorem{thm}{Theorem}[section]
\newtheorem{lem}[thm]{Lemma}

\newtheorem{remark}[thm]{Remark}

\newcommand{\R}{\mathbb{R}}

\newcommand{\wt}{\widetilde}

\begin{document}
\title[Asymptotic Behavior of Solutions for Competitive Models]{Asymptotic Behavior of Solutions for Competitive Models with
Free Boundaries$^\dag$}
 \thanks{$\dag$ This work is supported in part by NSFC (No. 11271285).}
\author[J. Yang]{Jian Yang$^\ddag$}
\thanks{$\ddag$ Department of Mathematics, Tongji University, Shanghai 200092, China.}
\thanks{{\bf Emails:} {\sf yangjian86419@126.com} (J. Yang). }

\begin{abstract}
In this paper, we study a competitive model involving two species. When the competition is strong enough, the two species are separated by a free boundary. If the initial data has a positive bound at infinity. We prove that the solution will converge, as $t\rightarrow \infty$, to the traveling wave solution and the free boundary will move to infinity with a constant speed.
\end{abstract}

\subjclass[2010]{35B40, 35K57, 35R35}
\keywords{Asymptotic behavior, competitive model, free boundary problem.}
\maketitle

\section{Introduction}
In this paper, we study the asymptotic behavior of solutions for the following competitive model
\begin{equation}\label{p}
\left\{
\begin{array}{l}
P_{t}=d_1P_{xx}+f(P),\hskip26mm x<s(t),\ t>0,\\
Q_{t}=d_2Q_{xx}+g(Q),\hskip25mm x>s(t),\ t>0,\\
P(x,t)=Q(x,t)=0, \hskip23mm  x=s(t),\ t>0,\\
s'(t)=-\mu_1 P_{x}(x,t)-\mu_2 Q_{x}(x,t), \ \ \ x=s(t),\ t>0,\\
s(0)=s_0,\ s_0\in(-\infty,\infty),\\
P(x,0)=P_0(x)(x<s_0),\ Q(x,0)=Q_0(x)(x>s_0)
\end{array}
\right.
\end{equation}
on unbounded domain, where $x=s(t)$ is the free boundary to be determined together with $P$ and $Q$, $f,g\in C^1([0,\infty))$
satisfying $f(0)=g(0)=0$.

We mainly consider monostable and bistable type of nonlinearities. More precisely, we call $f$ a monostable type of nonlinearity ($f$ is of (f$_M$) type, for short),
if $f\in C^1([0,\infty))$ and
\begin{equation*}
f(0)=0<f'(0), \quad f(1)=0> f'(1),\quad (1-s)f(s) >0 \ \mbox{for } s>0, s\not= 1;
\end{equation*}
we say that $f$ is a bistable type of nonlinearity ($f$ is of (f$_B$) type, for short),
if
\begin{equation*}
\hskip 10mm
\left\{
 \begin{array}{l}
 f\in C^1([0,\infty)), \  f(0)=0> f'(0), \ f(1)=0 > f'(1),\ \int_0^1 f(s) ds >0,\\
  \ f(\cdot) <0 \mbox{ in } (0,\theta)\cup (1,\infty),\medskip  f(\cdot)>0 \mbox{ in } (\theta,1) \mbox{ for some } \theta \in (0,1),
  \\ \mbox{ for } F(u):=-2\int_0^uf(s)ds,\ F(\bar{\theta})=0 \mbox{ for some } \bar{\theta}\in(\theta,1).
\end{array}
\right.
\end{equation*}

In population ecology, the appearance of regional partition of multi-species through strong competition
is one interesting phenomena. In \cite{MYY85,MYY86,MYY87}, Mimura, Yamada and Yotsutani used the following reaction-diffusion equations
\begin{equation}\label{MMY}
\left\{
\begin{array}{l}
P_{t}=d_1P_{xx}+f(P),\hskip26mm 0<x<s(t),\ t>0,\\
Q_{t}=d_2Q_{xx}+g(Q),\hskip25mm s(t)<x<1,\ t>0,\\
P(x,t)=Q(x,t)=0, \hskip23mm  x=s(t),\ t>0,\\
s'(t)=-\mu_1 P_{x}(x,t)-\mu_2 Q_{x}(x,t), \ \ \ x=s(t),\ t>0,\\
P(0,t)=m_1,\ Q(1,t)=m_2,\hskip13mm t>0,\\
s(0)=s_0(0<s_0<1),\\
P(x,0)=P_{0}(x)(0<x<s_0),\ Q(x,0)=Q_{0}(x)(s_0<x<1)
\end{array}
\right.
\end{equation}
to describe regional partition of two species, which are struggling on a boundary to obtain their own habitats. When $m_1,m_2>0$ and $f,g$ are monostable nonlinearities. They prove the global existence, uniqueness, regularity and asymptotic behavior of solutions for problem \eqref{MMY} in \cite{MYY85}. When $m_1,m_2>0$ and $f,g$ are bistable nonlinearities. The author establish the stability for stationary solutions for the free boundary problem and their argument is based on the notion of $\omega-$limit set and the comparison principle in \cite{MYY86}. When $m_1,m_2=0$, i.e. homogeneous Dirichlet boundary conditions are imposed, there is possibility that the free boundary may hit the fixed ends $x=0,1$ in a finite time. This is a very interesting phenomenon. Therefore, in \cite{MYY87}, the author prove that if the free boundary hits the fixed boundary at a finite time $t=T^*$, the free boundary stays there after $T^*$.

Problem \eqref{MMY} is defined on finite interval $[0,1]$. A natural question is what will happen if the two species competitive model is defined on $(-\infty,\infty)$? If the two species competitive model is defined on the entire space, the free boundary $s(t)$ could move to infinity in different ways and the problem may become more complicated. This is why we are interested in studying problem \eqref{p}. Possibly, the solution will develop into the traveling wave eventually if problem \eqref{MMY} is defined on unbounded domain. So, it is necessary to consider the traveling wave solution before we investigate the asymptotic behavior of solutions for problem \eqref{p}. To study the traveling wave solution of problem \eqref{p} is equivalent to study the solution of the following ordinary differential equations
\begin{equation}\label{erlianbo}
\left\{
\begin{array}{l}
d_1\phi''+c\phi'+f(\phi)=0,\quad x\in(-\infty,0],\\
d_2\psi''+c\psi'+g(\psi)=0,\quad x\in[0,\infty),\\
\phi(0)=\psi(0)=0,\\
\phi(-\infty)=\psi(+\infty)=1,\\
c=-\mu_1 \phi'(0)-\mu_2 \psi'(0).
\end{array}
\right.
\end{equation}
Recently, in \cite{CC}, Chang and Chen prove the existence of a traveling wave solution of \eqref{erlianbo} for logistic type nonlinearities. In \cite{YL}, we extend the results in \cite{CC} to more general nonlinearities. We prove that
if $\alpha>0$ is a given constant, then for any $c\in (c^*_g, \hat{c}_f)$,
where $c^*_g <0$ is the maximal speed when $g$ is of {\rm (f$_M$)} type, or the unique speed when $g$ is of {\rm (f$_B$)} type and $\hat{c}_f >0$ depends only on $\alpha$ and $f$, there exists a unique $\beta(c)>0$
such that \eqref{erlianbo} has a unique solution $(\phi, \psi, c)$.
Moreover, $\beta(c)$ is continuous and strictly decreasing in $c\in (c^*_g, \hat{c}_f)$ and
\begin{equation*}
c\to \hat{c}_f \ \Leftrightarrow \ \beta\to 0,\qquad
c\to c^*_g \ \Leftrightarrow \ \beta\to \infty,\qquad c>0\ \Leftrightarrow \ \beta< \tilde{\beta},
\end{equation*}
where $\tilde{\beta} := \alpha (\int_0^1f(s)ds / \int_0^1g(s)ds)^{1/2}$. If $\beta>0$ is a given constant, we can obtain similar results in the same method.

The main purpose of this paper is to prove that for any initial data $P_0\mbox{ and }Q_0$ in problem \eqref{p}, as long as they have positive lower bound at infinity, the solution of problem \eqref{p} will converge to a traveling wave as $t\to \infty$. Our main result is as follows:
\begin{thm}\label{thm:asymptoticbehavior} Assume $f,g$ are of (f$_M$) type and the initial data $P_0,Q_0$ satisfy
\begin{equation}\label{initialdataforfm}
0<\liminf_{x\rightarrow-\infty} P_0(x)<\limsup_{x\rightarrow-\infty} P_0(x)<\infty,\
0<\liminf_{x\rightarrow\infty} Q_0(x)<\limsup_{x\rightarrow\infty} Q_0(x)<\infty,
\end{equation}
then for some constant $x^*$, the solution of problem \eqref{p} satisfies
\begin{equation}\label{behaviorofuv}\left.
\begin{array}{l}
\displaystyle\lim_{t\rightarrow\infty}\sup_{(-\infty,s(t)]}|P(x,t)-\phi(x-ct-x^*)|=0, \
\displaystyle\lim_{t\rightarrow\infty}\sup_{[s(t),\infty)}|Q(x,t)-\psi(x-ct-x^*)|=0,
\end{array}\right.
\end{equation}
where $(\phi,\psi,c)$ is defined in \eqref{erlianbo}. Moreover,
\begin{equation}\label{asymptoticbehaviorofws}\left.
\begin{array}{c}
\displaystyle\lim_{t\rightarrow\infty}(s(t)-ct-x^*)=0, \ \lim_{t\rightarrow\infty}s'(t)= c.
\end{array}\right.
\end{equation}
\end{thm}

\begin{remark}\label{rem:rem1}
Similar results in Theorem \ref{thm:asymptoticbehavior} also holds when $f$ and/or $g$ are of (f$_B$) type; namely, (\ref{behaviorofuv}) and (\ref{asymptoticbehaviorofws}) also hold if the initial data $P,Q$ satisfy one of the following conditions:
\begin{equation*}
\begin{array}{l}
(i)\ \ \ \displaystyle 0<\liminf_{x\rightarrow-\infty} P_0(x)<\limsup_{x\rightarrow-\infty} P_0(x)<\infty,\
\theta<\liminf_{x\rightarrow\infty} Q_0(x)<\limsup_{x\rightarrow\infty} Q_0(x)<\infty, \\
\mbox{ when } f \mbox{ is of } (f_M) \mbox{ type} \mbox{ and } g \mbox{ is of } (f_B) \mbox{ type };\\
(ii)\ \ \displaystyle \theta<\liminf_{x\rightarrow-\infty} P_0(x)<\limsup_{x\rightarrow-\infty} P_0(x)<\infty,\
0<\liminf_{x\rightarrow\infty} Q_0(x)<\limsup_{x\rightarrow\infty} Q_0(x)<\infty, \\
\mbox{ when } f \mbox{ is of } (f_B) \mbox{ type} \mbox{ and } g \mbox{ is of } (f_M) \mbox{ type };\\
(iii)\ \displaystyle \theta<\liminf_{x\rightarrow-\infty} P_0(x)<\limsup_{x\rightarrow-\infty} P_0(x)<\infty,\
\theta<\liminf_{x\rightarrow\infty} Q_0(x)<\limsup_{x\rightarrow\infty} Q_0(x)<\infty,\\
\mbox{ when } f \mbox{ is of } (f_B) \mbox{ type} \mbox{ and } g \mbox{ is of } (f_B) \mbox{ type }.
\end{array}
\end{equation*}
\end{remark}

\begin{remark}\label{ly}
If the diffusion constants $d_1=d_2=1$, we can construct Lyapunov functional as in \cite{FM} to prove the asymptotic behavior of solutions for problem (\ref{p}). However, through out our paper, we assume that the diffusion constants $d_1\neq d_2$. Thus, the method used in \cite{FM} does not work in this case. It is essentially different from the case $d_1=d_2=1$. Therefore, in our paper, we use a different way to prove the asymptotic behavior of solutions.
\end{remark}

The contents of this paper will be organized as follows: In section \ref{sec:pre}, we give some basic results on global existence of smooth solution for problem \eqref{p} and comparison principle. In section \ref{sec:asymptoticbehavior}, we prove Theorem \ref{thm:asymptoticbehavior}.


\section{Preliminary results}\label{sec:pre}
\subsection{The global existence of solutions}\label{subsec:locglobal}
\begin{thm}\label{thm:existence1}
If the initial data $P_{0},Q_{0}$ satisfy
\begin{equation}\label{initialdata}
P_0\in C^2((-\infty,s_0]), \
Q_0\in C^2([s_0,\infty)),
\end{equation}
then problem \eqref{p} has a unique solution
$$(P,Q,s)\in C^{1+\alpha,(1+\alpha)/2}(\bar{\Omega}_1)\times C^{1+\alpha,(1+\alpha)/2}(\bar{\Omega}_2)\times C^{1+\alpha/2}([0,\infty)).$$
Moreover,
\begin{equation*}
\|P\|_{C^{1+\alpha,(1+\alpha)/2}(\bar{\Omega}_1)}+\|Q\|_{C^{1+\alpha,(1+\alpha)/2}(\bar{\Omega}_2)}
+\|s\|_{C^{1+\alpha/2}([0,\infty))}\leq C
\end{equation*}
where $\bar{\Omega}_1=\{(x,t):x\in(-\infty,s(t)],t>0\},\bar{\Omega}_2=\{(x,t):x\in [s(t),\infty),t>0\},0<\alpha<1,C$ depends only on $\|P_0\|_{C^{1+\alpha}((-\infty,s_0])} \mbox{ and } \|Q_0\|_{C^{1+\alpha}([s_0,\infty))}$.
\end{thm}

The proof of Theorem \ref{thm:existence1} will be given in the Appendix.

\begin{lem}\label{lem:stbounded}
Assume $(P,Q,s)$ be the solution of the free boundary problem \eqref{p} defined for $t\in(0,T)$ for some $T\in (0,\infty]$. Then there exists a positive constant $H$ independent of $T$ such that
\begin{equation*}
|s'(t)|\leq H  \mbox{ for }  0<t<T.
\end{equation*}
\end{lem}
\begin{proof}
For any fix $a\in (-\infty,\infty)$, with no loss of generality, we may assume $a\in[-l,l](l>0)$, then we consider the following problem
\begin{equation}\label{stationsolution}
\left\{
\begin{array}{l}
d_1U''+f(U)=0,\hskip14mm y<a,\ t>0,\\
d_2V''+g(V)=0,\hskip14mm y>a,\ t>0,\\
\mu_1U'(a)+\mu_2V'(a)=0,\quad  \  t>0,\\
U(a)=V(a)=0,\\
U(-l)=2L_U,\ V(l)=2L_V.
\end{array}
\right.
\end{equation}
where $L_U=\sup\{\|P(\cdot,t)\|_{L^\infty}:t>0\}\ L_V=\sup\{\|Q(\cdot,t)\|_{L^\infty}:t>0\}$.
We know $(U,V)$ is the solution of \eqref{stationsolution}. i.e., $(U,V)$ is a stationary solution of \eqref{p}. Choose $(U,V)$ satisfying $|U'(a)|>w_1, V'(a)>w_2,$ where $w_1=\sqrt{\frac{2}{d_1}\int_0^1f(s)ds},w_2=\sqrt{\frac{2}{d_2}\int_0^1g(s)ds}$, then $V'(a)=-\frac{\mu_1}{\mu_2}U'(a)>\frac{\mu_1}{\mu_2}w_1$, i.e., $V'(a)>\max\left\{w_2,\frac{\mu_1}{\mu_2}w_1\right\}$. We select $(U,V)$ satisfying $\displaystyle\min V'(a)\geq\max\left\{\max_{x\in[0,\infty)}Q'_0(x),w_2,\frac{\mu_1}{\mu_2}w_1\right\},\displaystyle\min|U'(a)|\geq\max\left\{\max_{x\in(-\infty,0]}P'_0(x),w_1\right\}$. Denote $|U'(a)|=:\alpha_0\in [\min|U'(a)|,\max|U'(a)|], \mbox{ correspondingly, } V'(a)=:\beta_0\in [\min V'(a),\max V'(a)]$.

{\bf Case \ 1}: $s(t)$ first moves across the point $(a,0)$ at $x$-axis from the right to the left, $Q(x,t)$ must cross with $V(x)$ at $(a,0)$. Therefore, $Q_x(s(t),t)\leq V'(a)=\beta_0, \mbox{ for }s'(t)\leq0, \mbox{ we have }-P_x(s(t),t)\leq-U'(a)=\alpha_0$, so $|s'(t)|\leq\mu_1|P_{x}(s(t),t)|+\mu_2|Q_{x}(s(t),t)|\leq\mu_1\alpha_0+\mu_2\beta_0=:H$.

{\bf Case \ 1-1}: $s(t)$ stop at the point $(a,0)$, we have $|s'(t)|\leq H$;

{\bf Case \ 1-2}: $s(t)$ moves backwards and then moves towards the left and cross the point $(a,0)$ again, this case can be discussed similarly as above;

{\bf Case \ 1-3}: $s(t)$ moves across the point $(a,0)$, then $P(x,t)$ must separate with $U(x)$ temporarily. Even if $s(t)$ moves towards the right and cross the point $(a,0)$ again, we can discuss this case as above, so $|s'(t)|\leq H$.

Since $a$ is arbitrary, so the conclusion of Lemma follows.
\end{proof}
\subsection{Comparison principle}
\begin{lem}\label{lem:lower}
Suppose that $f,g\in C^1([0,\infty))$ satisfying $f(0)=g(0)=0$, $T\in (0,\infty)$, $\underline{s}\in C^1([0,T])$, $\underline{u}\in C(\overline{D}_{1T})\cap C^{2,1}(D_{1T})$ with $D_{1T}=\{(x,t)\in\mathbb{R}^2:x\in (-\infty,\underline{s}(t)],\ 0<t\leq T\}$, $\underline{v}\in C(\overline{D}_{2T})\cap C^{2,1}(D_{2T})$ with $D_{2T}=\{(x,t)\in\mathbb{R}^2:x\in [\underline{s}(t),\infty),\ 0<t\leq T\}$ and
\begin{equation*}\left\{
\begin{array}{l}
\underline{u}_t\leq d_1\underline{u}_{xx}+f(\underline{u}), \hskip12mm x\in(-\infty,\underline{s}(t)),\ 0<t\leq T,\\
\underline{v}_t\geq d_2\underline{v}_{xx}+g(\underline{v}), \hskip13mm  x\in(\underline{s}(t),\infty),\ 0<t\leq T, \\
\underline{u}(\underline{s}(t),t)=\underline{v}(\underline{s}(t),t)=0,\ 0<t<T,\\
\underline{s}'(t)\geq-\mu_1\underline{u}_x(\underline{s}(t),t)-\mu_2\underline{v}_x(\underline{s}(t),t),\ 0<t<T. \\
\end{array}\right.
\end{equation*}
\end{lem}
If
$$\underline{u}(x,0)\leq P_0(x) \mbox{ in } (-\infty,s_0],\  \underline{v}(x,0)\geq Q_0(x) \mbox{ in } [s_0,\infty) \mbox{ and } \underline{s}_0\leq s_0,$$
where $(P,Q,s)$ is a solution of \eqref{p}, then
\begin{equation*}
\begin{array}{c}
\underline{u}(x,t)\leq P(x,t) \mbox{ for } x\in (-\infty,s(t)] \mbox{ and } t\in(0,T],\\
\underline{v}(x,t)\geq Q(x,t) \mbox{ for } x\in [s(t),\infty) \mbox{ and } t\in(0,T],\\
\underline{s}(t)\leq s(t) \mbox{ for }t\in(0,T].
\end{array}
\end{equation*}

The proof of Lemma \ref{lem:lower} is identical to that of Lemma 5.7 in \cite{DuLin}, so we omit the details here.
\begin{remark}\label{rem:low}
The triple $(\underline{u},\underline{v},\underline{s})$ is often called a lower solution of problem \eqref{p} on $[0,T]$ with initial data $(\underline{u}(x,0),\underline{v}(x,0),\underline{s}_0)$. An upper solution can be defined analogously by reversing all the inequalities.
\end{remark}

\section{The Proof of Theorem \ref{thm:asymptoticbehavior} }\label{sec:asymptoticbehavior}
In this section, $f,g$ are always assumed to be of (f$_M$) or (f$_B$) type.

\subsection{Converge to 1 uniformly at infinity}\label{subsec:sup}
We consider the solution of $d_1U''+f(U)=0$ with compact supports, if $f$ is of type (f$_M$) (resp. (f$_B$)), for each $m\in(0,1)$ (resp. $m\in(\bar{\theta},1)$), consider the trajectory $\Gamma$ given by $d_1U'^2=F(U)-F(m)$, which connect $(0,\sqrt{(F(0)-F(m))/d_1})$ and $(0,-\sqrt{(F(0)-F(m))/d_1})$ through $(m,0)$, and the solution $U_m$ satisfies $d_1U''_m+f(U_m)=0<U_m\leq m \mbox{ in }(-L_m,L_m),$ where
\begin{equation}\label{Lm}
L_m:=\int_0^m\frac{ds}{\sqrt{(F(s)-F(m))/d_1}},\ m\in(\bar{\theta},1).
\end{equation}

\begin{lem}\label{lem:positivebound}
(1) If $f,g$ are of (f$_M$) type and the initial data $P_0,Q_0$ satisfy \eqref{initialdataforfm}. $P$ is the solution of the problem
\begin{equation}\label{problemp}
\left\{
\begin{array}{l}
P_{t}-d_1P_{xx}=f(P),\hskip15mm x<s(t),\ t>0,\\
P(s(t),t)=0, \hskip23mm  \ t>0,\\
P(x,0)=P_0(x),\hskip20mm x\leq s_0.
\end{array}
\right.
\end{equation}
and $Q$ is the solution of the problem
\begin{equation}\label{problemq}
\left\{
\begin{array}{l}
Q_{t}-d_2Q_{xx}=g(Q),\hskip15mm x>s(t),\ t>0,\\
Q(s(t),t)=0, \hskip23mm  \ t>0,\\
Q(x,0)=Q_0(x),\hskip20mm x\geq s_0,
\end{array}
\right.
\end{equation}
then for any $p_0,q_0\in(0,1)$, there exist $T,M>0$ such that
\begin{equation*}
P(x,t)\geq1-\frac{3}{4}p_0,\ x\leq-M,\ t\geq T,\quad Q(x,t)\leq1+\frac{1}{2}q_0,\ x\geq M, \ t\geq T.
\end{equation*}

(2) If $f,g$ are of type (f$_B$) and the initial data $P_0,Q_0$ satisfy (iii) in Remark \ref{rem:rem1}. $P$ is the solution of \eqref{problemp}
and $Q$ is the solution of \eqref{problemq}, then for any $p_0,q_0\in(0,\frac{4}{3}(1-\bar{\theta}))$, there exist $T,M>0$ such that
\begin{equation*}
P(x,t)\geq1-\frac{3}{4}p_0,\ x\leq-M,\ t\geq T,\quad Q(x,t)\leq1+\frac{1}{2}q_0,\ x\geq M, \ t\geq T.
\end{equation*}
\end{lem}

\begin{proof} (1)
By Lemma \ref{lem:stbounded}, $|s'(t)|\leq H \mbox{ for } 0< t\leq T_1$, so we can choose $M_1>0$ large enough such that $\max_{t\in(0,T_1]}|s'(t)|T_1\leq M_1$. Since $\liminf_{x\rightarrow-\infty}u_0(x)>0$, so there exists $0<\sigma\ll1 \mbox{ and } M_2>M_1>0$ such that when $x<-M_2$, we have $u_0(x)\geq \sigma>0$. Let $\eta(t)$ be the solution of
$$\eta_t=f(\eta) \mbox{ on } [0,\infty), \ \eta(0)=\sigma.$$
Since $f(\cdot)>0 \mbox{ in }(0,1), \mbox{ for any }p_0\in(0,1)$ and $T_2=\int_\sigma^{1-\frac{1}{4}p_0}\frac{ds}{f(s)}$, we have $\eta(T_2)=1-\frac{1}{4}p_0$.

We fix $R=L_{1-\frac{3}{4}p_0}$. Let $L\gg R$ be a constant to be determined later(we may assume $L<9M_2$ with no loss of generality) and $w_0(x)$ be a function satisfying
\begin{equation*}
\begin{array}{c}
w_0(x)=\sigma \ \mbox{ when }x\in (-\infty,-10M_2+L-1],\ w_0(-10M_2+L)=0, \\
w'_0(x)\leq0 \ \mbox{ when }x\in (-10M_2+L-1,-10M_2+L).
\end{array}
\end{equation*}
Let $w(x,t)$ be the solution of the following problem
\begin{equation*}
\left\{
\begin{array}{l}
w_t-d_1w_{xx}=f(w),\hskip15mm \forall \ x\in(-\infty,-10M_2+L],\ t>0,\\
w(-10M_2+L,t)=0, \hskip12mm  \forall \ t>0,\\
w(x,0)=w_0(x),\hskip21mm \forall \ x\in(-\infty,-10M_2+L].
\end{array}
\right.
\end{equation*}
Set $\rho(x):=\frac{1}{1+(x+10M_2)^2}$ and $\zeta(x,t):=\rho(x)[w(x,t)-\eta(t)]$ satisfying $\zeta_t-d_1\zeta_{xx}=4d_1x\rho\zeta_x+[2d_1\rho+f']\zeta$. Denote $\Lambda:=2d_1+\max_{0\leq s\leq1}f'(s)$, so we have
$$\displaystyle\max_{x\in(-\infty,-10M_2+L]}\{\rho(x)|w(x,t)-\eta(t)|\}\leq e^{\Lambda t}\max_{x\in(-\infty,-10M_2+L]}\{\rho(x)|w_0(x)-\sigma|\}\leq\frac{e^{\Lambda t}}{1+(L-1)^2}.$$
Take $L=L(\sigma)=1+\sqrt{2(1+R^2)e^{\Lambda T_2}/p_0-1}$, when $x\in(-\infty,-10M_2+R]$, we have
$$|w(x,T_2)-\eta(T_2)|\leq \frac{1}{\rho(x)}\frac{e^{\Lambda T_2}}{1+(L-1)^2}\leq \frac{(1+R^2)e^{\Lambda T_2}}{1+(L-1)^2}=\frac{1}{2}p_0.$$
Consequently, we obtain $w(x,T_2)\geq\eta(T_2)-\frac{1}{2}p_0=1-\frac{3}{4}p_0$.

Since $P_0(x)\geq\sigma>0 \mbox{ for } x\in(-\infty,-10M_2+L]$, thus $P_0(x)\geq w_0(x) \mbox{ for } x\in(-\infty,-10M_2+L]$. By comparison principle, we have
\begin{equation*}
P(x,t)\geq w(x,t) \mbox{ for } x\in(-\infty,-10M_2+L], \ t\geq0.
\end{equation*}
In particular, we get $P(x,T_2)\geq w(x,T_2)=1-\frac{3}{4}p_0 \mbox{ for } x\in [-10M_2-R,-10M_2+R]$. By using the same method, we can repeat the above argument to prove $P(x,T_2)\geq w(x,T_2)=1-\frac{3}{4}p_0 \mbox{ for } x\in(-\infty,-10M_2-R] \mbox{ and } [-10M_2+R,-2M_2]$. Therefore, we have
\begin{equation*}
P(x,t)\geq1-\frac{3}{4}p_0, \ x\in (-\infty,-2M_2], \ t\geq T_2.
\end{equation*}

On the other hand, since $\limsup_{x\rightarrow\infty}Q_0(x)<\infty$, so there exist $K>0 \mbox{ and } M_3>M_1$ such that $|v_0(x)|\leq K \mbox{ for } x\in [M_3,\infty)$.
Let $\xi(t)$ be the solution of the following problem
\begin{equation}
\xi_t=g(\xi) \mbox{ on } [0,\infty),\ \xi(0)=\|v_0\|_{L^\infty}+1.
\end{equation}
Then $\xi(t)$ is an upper solution of \eqref{problemq}. So $Q(x,t)\leq \xi(t) \mbox{ for all }t\geq0$. Since $g(Q)<0 \mbox{ for } Q>1$, $\xi(t)$ is a decreasing function converging to 1 as $t\rightarrow\infty$. Thus, for any $q_0\in(0,1)$, there exists $T_3>0$ such that $\xi(t)\leq1+\frac{1}{2}q_0 \mbox{ for } t\geq T_3$. In particular, we have $Q(x,T_3)\leq \xi(T_3)<1+\frac{1}{2}q_0 \mbox{ for } x\in[M_3,\infty)$ and so
\begin{equation*}
Q(x,t)\leq1+\frac{1}{2}q_0 \mbox{ for } x\in[M_3,\infty),\ t\geq T_3.
\end{equation*}
Consequently, if we choose $T=\max\{T_1,T_2,T_3\} \mbox{ and }M>\max\{2M_2,M_3\}$, we obtain
\begin{equation*}
\begin{array}{l}
P(x,t)\geq1-\frac{3}{4}p_0, \ x\in(-\infty,-M], \ t\geq T,\quad
Q(x,t)\leq1+\frac{1}{2}q_0, \ x\in[M,\infty), \ t\geq T.
\end{array}
\end{equation*}

(2) If $f$ is of (f$_B$) type, since $\liminf_{x\rightarrow-\infty}u_0(x)>\theta$, so there exists $m\in(\theta,\bar{\theta}],M_2>0$ such that $u_0(x)\geq m>0 \mbox{ for } x<-M_2$. Let $\eta(t)$ be the solution of
$$\eta_t=f(\eta) \mbox{ on } [0,\infty), \ \eta(0)=m.$$
Since $f(\cdot)>0 \mbox{ in }(\theta,\bar{\theta}), \mbox{ for any }p_0\in(0,\frac{4}{3}(1-\bar{\theta})),\ $ and $T_2=\int_m^{1-\frac{1}{4}p_0}\frac{ds}{f(s)}$, we have $\eta(T_2)=1-\frac{1}{4}p_0$. Moreover, we know that $1-\frac{1}{4}p_0>\bar{\theta}+2\epsilon$ where $\epsilon=\frac{1-\bar{\theta}}{3}$. Here, we fix $R=L_{\bar{\theta}+\epsilon}$.
The following proof is the same as (1), so we omit the details. This completes the proof.
\end{proof}

\subsection{Precise estimates of the solutions in moving coordinates}\label{subsec:estimate}

In the following, we assume that $(P(x,t),Q(x,t))$ is the solution of \eqref{p}. Denote $u(z,t)=P(z+ct,t)=P(x,t),v(z,t)=Q(z+ct,t)=Q(x,t),$ which satisfies
\begin{equation*}
\left\{
\begin{array}{l}
u_{t}-d_1u_{zz}-cu_z=f(u),\hskip14mm z<s(t)-ct,\ t>0,\\
v_{t}-d_2v_{zz}-cv_z=g(v),\hskip15mm z>s(t)-ct,\ t>0.
\end{array}
\right.
\end{equation*}

\begin{lem}\label{lem:suplowsolution}
Under the assumption of Theorem \ref{thm:asymptoticbehavior}, there exists $T>0$ with some constants $\rho_1,\rho_2,p_0,q_0\mbox{ and }\vartheta$ such that
\begin{equation}\label{asymptoticbehaviorofu}\left.
\begin{array}{l}
\phi(z-\rho_1)-p_0e^{-\vartheta t}\leq u(z,t)\leq\phi(z-\rho_2)+p_0e^{-\vartheta t},
\end{array}\right.
\end{equation}
\begin{equation}\label{asymptoticbehaviorofv}\left.
\begin{array}{l}
\psi(z-\rho_2)-q_0e^{-\vartheta t}\leq v(z,t)\leq\psi(z-\rho_1)+q_0e^{-\vartheta t},\\
\end{array}\right.
\end{equation}
for all $z$ and $t\geq T$. Moreover, $|s(t)-ct|$ is bounded for all $t>0$.
\end{lem}
\begin{proof}
Firstly, we prove the left part of \eqref{asymptoticbehaviorofu} and the right part of \eqref{asymptoticbehaviorofv}.

We define $\underline{u}(z,t)=\max\{0,\phi(z-\xi(t)+\alpha(t))-p(t)\},\underline{v}(z,t)=\psi(z+\eta(t)-\beta(t))+q(t)$ and say $\underline{z}(t)$ the free boundary of $(\underline{u},\underline{v})$. By Lemma \ref{lem:lower}, if $(\underline{u},\underline{v},\underline{z})$ is a lower solution of $(u,v,s)$, it has to satisfy
\begin{equation}\label{subsolution}\left\{
\begin{array}{l}
\underline{u}_t-d_1\underline{u}_{zz}-c\underline{u}_z-f(\underline{u})\leq0, \ z\in(-\infty,\underline{z}(t)),\\
\underline{v}_t-d_2\underline{v}_{zz}-c\underline{v}_z-g(\underline{v})\geq0, \ z\in(\underline{z}(t),\infty),\\
\underline{u}(\underline{z}(t),t)=\underline{v}(\underline{z}(t),t)=0,\\
\underline{z}'(t)\leq-\mu_1\underline{u}_z(\underline{z}(t),t)-\mu_2\underline{v}_z(\underline{z}(t),t)-c,\\
\underline{u}(z,0)\leq u_0(z),\ \underline{v}(z,0)\geq v_0(z),\ \underline{z}(0)\leq s(0).
\end{array}\right.
\end{equation}
We will check \eqref{subsolution} one by one. Denote
$$\aligned
A[\underline{u}]:=&\underline{u}_t-d_1\underline{u}_{zz}-c\underline{u}_z-f(\underline{u})\\
=&\phi'(\xi'+\alpha')-p'-d_1\phi''-c\phi'-f(\phi-p)\\
=&-\xi'\phi'+\alpha'\phi'-p'+f(\phi)-f(\phi-p).
\endaligned$$
Here we assume that $\xi'<0,\alpha'>0$, since $f'(1)<0$, we choose $\varrho>0,0<\delta<1$ such that $f(u)-f(u-p)\leq-\varrho p \mbox{ for } |u-1|<\delta \mbox{ and }|p|<\delta$, when $\phi\in [1-\delta,1)$,
$$
A[\underline{u}]\leq-p'+f(\phi)-f(\phi-p)=-p'+f'(\varpi)p\leq-p'-\varrho p,
$$
we can choose $p=p_0e^{-\varrho t}\mbox{ and }p_0<\delta$ such that $A[\underline{u}]\leq0$. When $\phi\in [0,1-\delta)$, we choose $\phi'\leq-\gamma \mbox{ and } r>0$, then we have
$$
\aligned
A[\underline{u}]=&-\xi'\phi'+\alpha'\phi'-p'+f(\phi)-f(\phi-p)\\
\leq&\gamma\xi'-p'+f'(\varsigma)p\\
\leq&\gamma\xi'-p'+rp.
\endaligned
$$
Therefore, we choose $\xi'=\frac{p'-rp}{\gamma}=-\frac{\varrho+r}{\gamma}p(<0)\mbox{ such that } A[\underline{u}]\leq0. \mbox{ Thus } \xi(t)=z_1+z_2e^{-\varrho t}, \mbox{ where } z_1=\xi(0)-\frac{\varrho+r}{\varrho \gamma}p_0, \ z_2=\frac{\varrho+r}{\varrho \gamma}p_0$ and $\xi(0)$ is a constant to be determined later.

Next, we extend the domain of $\psi$ from $[0,\infty)$ to $(-\infty,\infty)$, define
$$
\psi(x)=\left\{\aligned&\psi(x),  \hskip22mm x\in[0,\infty),\\
&\lambda_0-\lambda_0e^{-\frac{\psi'(0)}{\lambda_0}x}, \quad x\in(-\infty,0],
\endaligned\right.
$$
where $\lambda_0>0$ and it satisfies $\lambda_0c<\psi'(0)d_2$.
$$\aligned
B[\underline{v}]:=&\underline{v}_t-d_2\underline{v}_{zz}-c\underline{v}_z-g(\underline{v})\\
=&\psi'(\eta'-\beta')+q'-d_2\psi''-c\psi'-g(\psi+q)\\
=&\eta'\psi'-\beta'\psi'+q'+g(\psi)-g(\psi+q).
\endaligned$$
We assume that $\eta'>0,\beta'<0$. Since $g'(1)<0$, so there exist $\vartheta,\nu>0$ such that $g(v)-g(v+q)\geq\vartheta q \mbox{ for } |v-1|<\nu \mbox{ and } |q|<\nu$. Here we choose $q_0<\nu\mbox{ and }\vartheta<\varrho$,  when $\psi\in [1-\nu,1)$,
$$
B[\underline{v}]\geq q'+g(\psi)-g(\psi+q)\geq q'+\vartheta q,
$$
so we choose $q=q_0e^{-\vartheta t}$ such that $B[\underline{v}]\geq0$. When $\psi\in [0,1-\nu)$, we select $\psi'\geq\lambda>0,\ \tau>0 \mbox{ and }\lambda<\frac{(\vartheta+\tau)q_0\lambda_0}{\mu_2\psi'(0)q_0+\lambda_0(\mu_1\phi'(\phi^{-1}(p_0))-\mu_1\phi'(0))}$, then we obtain
$$
\aligned
B[\underline{v}]=&\eta'\psi'-\beta'\psi'+q'+g(\psi)-g(\psi+q)\\
\geq&\eta'\psi'+q'+g(\psi)-g(\psi+q)\\
=&\eta'\psi'+q'-g'(\upsilon)q\\
\geq&\eta'\psi'+q'-\tau q\\
\geq&\lambda\eta'+q'-\tau q.
\endaligned
$$
So we choose $\eta'=\frac{-q'+\tau q}{\lambda}=\frac{\vartheta+\tau}{\lambda}q(>0)\mbox{ such that } B[\underline{v}]\geq0.\mbox{ Thus }\eta(t)=z_3-z_4e^{-\vartheta t}, \mbox{ where } z_3=\eta(0)+\frac{\vartheta+\tau}{\vartheta \lambda}q_0, \ z_4=\frac{\vartheta+\tau}{\vartheta \lambda}q_0$ and $\eta(0)$ is a constant to be determined later.

Denote $\mathbf{h}(z,t):=z+\eta(t)-\beta(t), \ w(z,t):=\lambda_0-\lambda_0e^{-\frac{\psi'(0)}{\lambda_0}\mathbf{h}}+q(t)$. $g(w)=g(0)+g'(\zeta)w$, we denote $l_0:=\max_{0\leq\zeta\leq q_0}|{g'(\zeta)}|$ , so we have $g(w)\leq l_0w\leq l_0q$.
$$
\aligned
B[w]=&w_t-d_2w_{zz}-cw_z-g(w)\\
=&\psi'(0)e^{-\frac{\psi'(0)}{\lambda_0}\mathbf{h}}[\eta'-\beta']+q'+d_2\frac{\psi'^2(0)}{\lambda_0}e^{-\frac{\psi'(0)}{\lambda_0}\mathbf{h}}
-c\psi'(0)e^{-\frac{\psi'(0)}{\lambda_0}\mathbf{h}}-g(w)\\
\geq&q'+d_2\frac{\psi'^2(0)}{\lambda_0}e^{-\frac{\psi'(0)}{\lambda_0}\mathbf{h}}-c\psi'(0)e^{-\frac{\psi'(0)}{\lambda_0}\mathbf{h}}-l_0q\\
=&-(\vartheta+l_0)q+\left(\frac{\psi'^2(0)d_2}{\lambda_0}-c\psi'(0)\right)e^{-\frac{\psi'(0)}{\lambda_0}\mathbf{h}}.
\endaligned
$$
Since $q(t)$ is bounded and $\lambda_0c<\psi'(0)d_2$, we can choose $q_0\leq\min\left\{\frac{\psi'^2(0)d_2-c\lambda_0\psi'(0)}{\lambda_0(\vartheta+l_0)},\frac{4}{3}(1-\bar{\theta})\right\}$ small such that $B[w]\geq0$.

In the following, we select $\alpha(t)=\eta(t)+\phi^{-1}(p(t)),\ \beta(t)=\xi(t)+\lambda_0(\psi'(0))^{-1}\log\left(1+\frac{q(t)}{\lambda_0}\right)$, then we have $\underline{z}(t)=\xi(t)-\eta(t)$.
$$
\aligned
w(\underline{z}(t),t)&=\lambda_0-\lambda_0e^{-\frac{\psi'(0)}{\lambda_0}(\xi(t)-\eta(t)+\eta(t)-\beta(t))}+q(t)\\
&=\lambda_0-\lambda_0e^{-\frac{\psi'(0)}{\lambda_0}\left(-\lambda_0(\psi'(0))^{-1}\log\left(1+\frac{q(t)}{\lambda_0}\right)\right)}+q(t)\\
&=\lambda_0-\lambda_0\left(1+\frac{q(t)}{\lambda_0}\right)+q(t)=0,
\endaligned
$$
and
$$
\aligned
\phi(\underline{z}(t),t)-p(t)&=\phi(\xi(t)-\eta(t)-\xi(t)+\alpha(t))-p(t)\\
&=\phi(\alpha(t)-\eta(t))-p(t)\\
&=\phi(\phi^{-1}(p(t)))-p(t)=0.
\endaligned
$$

$$\underline{z}'(t)=\xi'(t)-\eta'(t)=-\frac{\varrho+r}{\gamma}p_0e^{-\varrho t}-\frac{\vartheta+\tau}{\lambda}q_0e^{-\vartheta t}
<-\frac{\vartheta+\tau}{\lambda}q_0e^{-\vartheta t}, $$\ $$-\mu_1\underline{u}_z(\underline{z}(t),t)-\mu_2\underline{v}_z(\underline{z}(t),t)-c=-\mu_1\phi'(\phi^{-1}(p(t)))+\mu_1\phi'(0)-\mu_2\psi'(0)\lambda_0^{-1}q_0e^{-\vartheta t}.$$
Since $\lambda<\frac{(\vartheta+\tau)q_0\lambda_0}{\mu_2\psi'(0)q_0+\lambda_0(\mu_1\phi'(\phi^{-1}(p_0))-\mu_1\phi'(0))}$, we have
$$\underline{z}'(t)\leq-\mu_1\underline{u}_z(\underline{z}(t),t)-\mu_2\underline{v}_z(\underline{z}(t),t)-c.$$

Denote $$\rho^*:=-\eta(0)+\xi(0)-\phi^{-1}(p_0),\ \varrho^*:=-\eta(0)+\xi(0)+\lambda_0(\psi'(0))^{-1}\log\left(1+\frac{q_0}{\lambda_0}\right),$$ $$\Sigma:=\xi(0)-\frac{\varrho+r}{\varrho \gamma}p_0-\eta(0)-\frac{\vartheta+\tau}{\vartheta \lambda}q_0,\ \rho_1:=-\eta(0)+\xi(0)-\frac{\varrho+r}{\varrho \gamma}p_0-\frac{\vartheta+\tau}{\vartheta \lambda}q_0.$$
By Lemma \ref{lem:positivebound}, we can select $\eta(0)$ positive and large enough while $\xi(0)$ negative such that

$$\phi(z-\rho^*)-p_0\leq 1-\frac{3}{4}p_0\leq u(z,t) \mbox{ for } z\in(-\infty,-M],\ t\geq T,$$

$$\psi(z-\varrho^*)+q_0\geq 1+\frac{1}{2}q_0\geq v(z,t) \mbox{ for } z\in[M,\infty),\ t\geq T.$$
Thus, we know that $s(t)\geq \underline{z}(t)+ct,$ and
$\displaystyle\liminf_{t\rightarrow\infty}(s(t)-ct)\geq\Sigma.$
Consequently, $(\underline{u},\underline{v},\underline{z})$ is a lower solution of $(u,v,s)$ for $t\geq T.$ It follows that
$$u(z,t)\geq\phi(z-\xi(t)+\alpha(t))-p(t)\geq\phi(z-\rho_1)-p_0e^{\vartheta t}$$
and
$$v(z,t)\leq\psi(z+\eta(t)-\beta(t))+q(t)\leq\psi(z-\rho_1)+q_0e^{\vartheta t}.$$
Using the same method, we can construct an upper solution in the similar way. The proof is now completed.
\end{proof}

\begin{lem}\label{lem:closeaftert}
Under the assumption of Theorem \ref{thm:asymptoticbehavior}, there exists a function $\omega(\varepsilon)$, defined for small positive $\varepsilon$, satisfying $\lim_{\varepsilon\rightarrow0}\omega(\varepsilon)=0$. And if there exists $T>0$ such that $|u(z,T)-\phi(z-\rho_0)|<\varepsilon,\ |v(z,T)-\psi(z-\rho_0)|<\varepsilon$ for some $\rho_0$, then
$$|u(z,t)-\phi(z-\rho_0)|<\omega(\varepsilon),\ |v(z,t)-\psi(z-\rho_0)|<\omega(\varepsilon)$$
for all $z$ and all $t\geq T$.
\end{lem}
\begin{proof}
In the proof of Lemma \ref{lem:suplowsolution}, we may take $p_0=O(\varepsilon),\ q_0=O(\varepsilon),\ |\rho^*-\rho_0|=O(\varepsilon) \mbox{ and } |\varrho^*-\rho_0|=O(\varepsilon)$. Hence also $|\rho_1-\rho_0|=O(\varepsilon),\ |\rho_2-\rho_0|=O(\varepsilon)$ and the conclusion follows from Lemma \ref{lem:suplowsolution}.
\end{proof}

\subsection{Convergence of the solutions}
By Lemma \ref{lem:suplowsolution}, we know that there exist $C,T>0$ such that
$$-C\leq s(t)-ct\leq C \mbox{ for } t\geq T.$$
Here we assume that $C>\max\{\rho_1,\rho_2\}$.
Denote
$$k(t)=ct-2C$$
and define
$$\widetilde{u}(x,t)=P(x+k(t),t),\widetilde{v}(x,t)=Q(x+k(t),t),\widetilde{s}(t)=s(t)-k(t), t\geq T.$$
Obviously,
$$C\leq \widetilde{s}(t)\leq 3C \mbox{ for } t\geq T.$$
By simple calculation, we know that $(\widetilde{u},\widetilde{v},\widetilde{s})$ satisfies
\begin{equation*}\left\{
\begin{array}{l}
\widetilde{u}_t-d_1\widetilde{u}_{xx}-c\widetilde{u}_x-f(\widetilde{u})=0, \quad x\in(-\infty,\widetilde{s}(t)),t\geq T,\\
\widetilde{v}_t-d_2\widetilde{v}_{xx}-c\widetilde{v}_x-g(\widetilde{v})=0, \ \quad x\in(\widetilde{s}(t),\infty),t\geq T,\\
\widetilde{u}(\widetilde{s}(t),t)=\widetilde{v}(\widetilde{s}(t),t)=0,\ \quad \quad t\geq T,\\
\widetilde{s}'(t)=-\mu_1\widetilde{u}_x(\widetilde{s}(t),t)-\mu_2\widetilde{v}_x(\widetilde{s}(t),t)-c,t\geq T.\\
\end{array}\right.
\end{equation*}
Let $t_n\rightarrow\infty$ be an arbitrary sequence satisfying $t_n\geq T$ for every $n\geq 1$.

Define
$$k_n(t)=k(t+t_n),\ \widetilde{u}_n(x,t)=\widetilde{u}(x,t+t_n),\ \widetilde{v}_n(x,t)=\widetilde{v}(x,t+t_n),\ \widetilde{s}_n(t)=\widetilde{s}(t+t_n).$$
\begin{lem}\label{lem:sequencebounded} Subject to a subsequence,
\begin{equation*}
\widetilde{s}_n(t)\rightarrow G(t) \mbox{ in } C_{loc}^{1+\alpha/2}(\mathbb{R}^1), \widetilde{u}_n\rightarrow \widetilde{U} \mbox { in  } C_{loc}^{1+\alpha,(1+\alpha)/2}(D_u) \mbox{ and }\widetilde{v}_n\rightarrow \widetilde{V} \mbox { in } C_{loc}^{1+\alpha,(1+\alpha)/2}(D_v),
\end{equation*}
where $0<\alpha<1, D_u=\{(x,t):-\infty<x<G(t),t\in \mathbb{R}^1\},D_v=\{(x,t): G(t)<x<\infty,t\in \mathbb{R}^1\}$ and $(\widetilde{U}(x,t),\widetilde{V}(x,t),G(t))$ satisfies
\begin{equation}\label{lemequation1}\left\{
\begin{array}{l}
\widetilde{U}_t-d_1\widetilde{U}_{xx}-c\widetilde{U}_x=f(\widetilde{U}), \quad (x,t)\in D_u,\\
\widetilde{V}_t-d_2\widetilde{V}_{xx}-c\widetilde{V}_x=g(\widetilde{V}), \quad (x,t)\in D_v,\\
\widetilde{U}(G(t),t)=\widetilde{V}(G(t),t)=0,  \quad   t\in \mathbb{R}^1, \\
G'(t)=-\mu_1 \widetilde{U}_x(G(t),t)-\mu_2 \widetilde{V}_x(G(t),t)-c,\quad  t\in \mathbb{R}^1.
\end{array}\right.
\end{equation}
\end{lem}
\begin{proof}
By Lemma \ref{lem:stbounded}, we have $|s'(t)|\leq H,$ for all $t>0$. So, there exists $\widetilde{H}>0$ such that $|\widetilde{s}'_n(t)|\leq \widetilde{H} \mbox{ for } t+t_n \mbox{ large and every } n\geq1.$

Define
$$y=\frac{x}{\widetilde{s}_n(t)},\hat{u}_n(y,t)=\widetilde{u}_n(x,t),\hat{v}_n(y,t)=\widetilde{v}_n(x,t),$$
then $(\hat{u}_n(y,t),\hat{v}_n(y,t),\widetilde{s}_n(t))$ satisfies
\begin{equation}\label{lemequation2}\left\{
\begin{array}{l}
(\hat{u}_n)_t-\frac{d_1}{\widetilde{s}^2_n(t)}(\hat{u}_n)_{yy}-\big[y\widetilde{s}'_n(t)+c\big]\frac{(\hat{u}_n)_y}{\widetilde{s}_n(t)}=f(\hat{u}_n),
 \quad -\infty< y\leq1,t>T-t_n,\\
(\hat{v}_n)_t-\frac{d_2}{\widetilde{s}^2_n(t)}(\hat{v}_n)_{yy}-\big[y\widetilde{s}'_n(t)+c\big]\frac{(\hat{v}_n)_y}{\widetilde{s}_n(t)}=g(\hat{v}_n),
\  \quad 1\leq y<\infty,t>T-t_n,\\
\hat{u}_n(1,t)=\hat{v}_n(1,t)=0, \hskip45mm    t>T-t_n, \\
\widetilde{s}_n'(t)=-\mu_1 \frac{(\hat{u}_n)_y(1,t)}{\widetilde{s}_n(t)}-\mu_2 \frac{(\hat{v}_n)_y(1,t)}{\widetilde{s}_n(t)}-c, \hskip18mm t>T-t_n.
\end{array}\right.
\end{equation}
For any given $R_0>0$ and $T_0\in\R^1$, by using the interior-boundary $L^p$ estimates (see Theorem 7.15 in \cite{L}) to \eqref{lemequation2} over $[-R_0-1,1]\times[T_0-1,T_0+1]\mbox{ and }[1,R_0+1]\times[T_0-1,T_0+1]$, we get, for any $p>1$
$$\|\hat{u}_n\|_{W_p^{1,2}([-R_0,1]\times[T_0,T_0+1])}\leq C_{R_0}\mbox{ for all large }n,$$
$$\|\hat{v}_n\|_{W_p^{1,2}([1,R_0]\times[T_0,T_0+1])}\leq C_{R_0}\mbox{ for all large }n,$$
where $C_{R_0}$ is a constant depending on $R_0$ and $p$ but independent of $n$ and $T_0$. Therefore, for any $\alpha'\in(0,1)$, we can choose $p>1$ large enough and use the Sobolev embedding theorem (see \cite{LSU}) to obtain
\begin{equation}\label{lemequation3}\left.
\begin{array}{l}
\|\hat{u}_n\|_{C^{1+\alpha',\frac{1+\alpha'}{2}}([-R_0,1]\times[T_0,\infty))}\leq \widetilde{C}_{R_0}\mbox{ for all large }n,\\
\|\hat{v}_n\|_{C^{1+\alpha',\frac{1+\alpha'}{2}}([1,R_0]\times[T_0,\infty))}\leq \widetilde{C}_{R_0}\mbox{ for all large }n,
\end{array}\right.
\end{equation}
where $\widetilde{C}_{R_0}$ is a constant depending on $R_0$ and $\alpha'$ but independent of $n$ and $T_0$.
From \eqref{lemequation2} and \eqref{lemequation3}, we deduce that
$$\|\widetilde{s}_n\|_{C^{1+\frac{\alpha'}{2}}([T_0,\infty))}\leq C_1\mbox{ for all large }n,$$
where $C_1$ is a constant independent of $n$ and $T_0$.
Hence by passing to a subsequence, we obtain that, as $n\rightarrow\infty$,
$$\hat{u}_n\rightarrow\hat{U}\mbox{ in } C^{1+\alpha,\frac{1+\alpha}{2}}_{loc}((-\infty,1]\times\R^1),\quad \hat{v}_n\rightarrow\hat{V}\mbox{ in } C^{1+\alpha,\frac{1+\alpha}{2}}_{loc}([1,\infty)\times\R^1),$$
$$\widetilde{s}_n\rightarrow G\mbox{ in } C^{1+\frac{\alpha}{2}}_{loc}(\R^1),$$
where $\alpha\in (0,\alpha')$. Moreover, by using \eqref{lemequation2}, we know that $(\hat{U},\hat{V},G)$ satisfies in the $W_p^{1,2}$ sense (and hence classical sense by standard regularity theory),
\begin{equation*}\left\{
\begin{array}{l}
(\hat{U})_t-\frac{d_1}{G^2(t)}(\hat{U})_{yy}-\big[yG'(t)+c\big]\frac{(\hat{U})_y}{G(t)}=f(\hat{U}),
 \quad -\infty< y\leq1,t\in \R^1,\\
(\hat{V})_t-\frac{d_2}{G^2(t)}(\hat{V})_{yy}-\big[yG'(t)+c\big]\frac{(\hat{V})_y}{G(t)}=g(\hat{V}),
\  \quad 1\leq y<\infty,t\in\R^1,\\
\hat{U}(1,t)=\hat{V}(1,t)=0, \hskip45mm    t\in \R^1, \\
G'(t)=-\mu_1 \frac{(\hat{U})_y(1,t)}{G(t)}-\mu_2 \frac{(\hat{V})_y(1,t)}{G(t)}-c, \hskip18mm t\in \R^1.
\end{array}\right.
\end{equation*}
Define $\widetilde{U}(x,t)=\hat{U}(\frac{x}{G(t)},t), \widetilde{V}(x,t)=\hat{V}(\frac{x}{G(t)},t)$, then $(\widetilde{U},\widetilde{V},G)$ satisfies  \eqref{lemequation1} and
$$\lim_{n\rightarrow\infty}\|\widetilde{u}_n- \widetilde{U}\|_{C_{loc}^{1+\alpha,(1+\alpha)/2}(D_u)}=0,$$
$$\lim_{n\rightarrow\infty}\|\widetilde{v}_n- \widetilde{V}\|_{C_{loc}^{1+\alpha,(1+\alpha)/2}(D_v)}=0.$$
\end{proof}
Since $C\leq \widetilde{s}(t)\leq 3C\mbox{ for }t\geq T$, so
$$C\leq G(t)\leq 3C \mbox{ for } t\in \R^1.$$
By the proof of Lemma \ref{lem:suplowsolution}, we have
$$\widetilde{u}_n(x,t)\leq\phi(x-2C-\rho_2)+p_0e^{-\vartheta(t+t_n)},$$
$$\widetilde{v}_n(x,t)\geq\psi(x-2C-\rho_2)-q_0e^{-\vartheta(t+t_n)}.$$
Letting $n\rightarrow\infty$, we obtain
$$\widetilde{U}(x,t)\leq\phi(x-2C-\rho_2)\mbox{ for all }t\in\R^1,x<G(t),$$
$$\widetilde{V}(x,t)\geq\psi(x-2C-\rho_2)\mbox{ for all }t\in\R^1,x>G(t).$$
Define
$$R_u=\inf\{R:\widetilde{U}(x,t)\leq \phi(x-R) \mbox{ for } (x,t)\in D_u\},$$
$$R_v=\inf\{R:\widetilde{V}(x,t)\geq \psi(x-R) \mbox{ for } (x,t)\in D_v\},$$
then
$$\widetilde{U}(x,t)\leq \phi(x-R_u) \mbox{ for } (x,t)\in D_u,\widetilde{V}(x,t)\geq \psi(x-R_v)\mbox{ for } (x,t)\in D_v.$$
Denote $R_*=\max\{R_u,R_v\}$, hence
$$\widetilde{U}(x,t)\leq \phi(x-R_*) \mbox{ for } (x,t)\in D_u,\widetilde{V}(x,t)\geq \psi(x-R_*)\mbox{ for } (x,t)\in D_v$$
and
$$C\leq\displaystyle\inf_{t\in\R^1} G(t)\leq\sup_{t\in\R^1}G(t)\leq R_*\leq 3C.$$

\begin{lem}\label{lem:sequence1}
$R_*=\sup_{t\in \mathbb{R}^1}G(t)$.
\end{lem}
\begin{proof}
We prove it by contradiction, if not, we have $R_*>\sup_{t\in \R^1}G(t)$. Choose $\delta>0$ such that $G(t)\leq R_*-\delta \mbox{ for all } t\in \R^1$.

\textit{Step 1:} $\widetilde{U}(x,t)<\phi(x-R_*) \mbox{ for all } t\in \R^1 \mbox{ and } x\leq G(t),\widetilde{V}(x,t)>\psi(x-R_*) \mbox{ for all } t\in \R^1 \mbox{ and } x\geq G(t)$.

Otherwise, there exists $(x_0,t_0)\in D_u=\{(x,t):-\infty<x\leq G(t),t\in\R^1\}$ such that $\widetilde{U}(x_0,t_0)=\phi(x_0-R_*)\geq\phi(-\delta)>0$, therefore, $x_0<G(t_0)$. Since $\widetilde{U}(x,t)\leq\phi(x-R_*) \mbox{ in } D_u$ and $\phi(x-R_*)$ satisfies the equation \eqref{lemequation1}$_1$, by strong maximum principle, we conclude that $\widetilde{U}(x,t)\equiv\phi(x-R_*)$ in $D_{0u}:=\{(x,t):x<G(t),t\leq t_0\}$, and this contradicts with $G(t)\leq R_*-\delta$.
If there exists $(x_0,t_0)\in D_v=\{(x,t):G(t)\leq x<\infty, t\in \R^1\}$ such that $\widetilde{V}(x_0,t_0)=\psi(x_0-R_*)$.

{\bf Case 1} $x_0=R_*$. We have $\widetilde{V}(x_0,t_0)=\psi(0)=0$, so $x_0=R_*=G(t_0)$, and this contradiction implies the conclusion easily.

{\bf Case 2} $x_0>R_*$. We have $\widetilde{V}(x_0,t_0)=\psi(x_0-R_*)>\psi(0)=0$. Since $\widetilde{V}(x,t)\geq\psi(x-R_*) \mbox{ in } D_v$ and $\psi(x-R_*)$ satisfies the equation \eqref{lemequation1}$_2$, by strong maximum principle, we conclude that $\widetilde{V}(x,t)\equiv\psi(x-R_*)$ in $D_{0v}:=\{(x,t):x>G(t),t\leq t_0\}$, so $\widetilde{V}(R_*,t)=\psi(R_*-R_*)=0 \mbox{ for } t\leq t_0$ and $R_*=G(t) \mbox{ for }t\leq t_0$. This is a contradiction with $G(t)\leq R_*-\delta$.

\textit{Step 2:} $M_u(x):=\inf_{t\in\R^1}[\phi(x-R_*)-\widetilde{U}(x,t)]>0\mbox{ for } x\in (-\infty,R_*-\delta]$  and $M_v(x):=\inf_{t\in\R^1}[\widetilde{V}(x,t)-\psi(x-R_*)]>0\mbox{ for } x\in [R_*,\infty)$.
Otherwise, there exist $x_{10}\in (-\infty,R_*-\delta]$ and $x_{20}\in [R_*,\infty)$ such that $M_u(x_{10})=0,M_v(x_{20})=0$, since the definition of $R_*$ implies $M_u(x)\geq0$. By step 1, we know $M_u(x_{10})$ and $M_v(x_{20})$ can not be achieved at any finite $t$. Therefore, there exists $s_n\in\R^1$ with $|s_n|\rightarrow\infty$ such that
$$\phi(x_{10}-R_*)=\lim_{n\rightarrow\infty}\widetilde{U}(x_{10},s_n).$$
Define
$$(\widetilde{U}_n(x,t),\widetilde{V}_n(x,t),G_n(t))=(\widetilde{U}(x,t+s_n),\widetilde{V}(x,t+s_n),G(t+s_n)).$$
Then we can use the same argument as in the proof of Lemma \ref{lem:sequencebounded} to show that, by passing to a subsequence, $(\widetilde{U}_n,\widetilde{V}_n,G_n)\rightarrow(U^*,V^*,G^*)$ with $(U^*,V^*,G^*)$ satisfying
\begin{equation}\label{lemequation4}\left\{
\begin{array}{l}
U^*_t-d_1U^*_{xx}-cU^*_x=f(U^*), \quad -\infty<x<G^*(t),\ t\in\R^1,\\
V^*_t-d_2V^*_{xx}-cV^*_x=g(V^*), \quad G^*(t)<x<\infty,\ t\in\R^1,\\
U^*(G^*(t),t)=V^*(G^*(t),t)=0,  \quad   t\in \mathbb{R}^1. \\
\end{array}\right.
\end{equation}
Moreover,
\begin{equation}\label{lemequation5}\left.
\begin{array}{c}
U^*(x,t)\leq\phi(x-R_*),\ V^*(x,t)\geq\psi(x-R_*),\ G^*(t)\leq R_*-\delta,\\
\ U^*(x_{10},0)=\phi(x_{10}-R_*),\ V^*(x_{20},0)=\psi(x_{20}-R_*).
\end{array}\right.
\end{equation}
Since $(\phi(x-R_*),\psi(x-R_*))$ satisfies \eqref{lemequation4} with $G^*(t)$ replaced by $R_*$, by strong maximum principle, we have
$U^*(x,t)\equiv\phi(x-R_*)\mbox{ for } t\leq0, \ x\leq G^*(t),\ V^*(x,t)\equiv\psi(x-R_*) \mbox{ for }t\leq0, \ x\geq G^*(t)$, which is clearly impossible. On the other hand, If there exists $\tau_n\in\mathbb{R}$ with $|\tau_n|\rightarrow\infty$ such that  $\psi(x_{20}-R_*)=\displaystyle\lim_{n\rightarrow\infty}\widetilde{V}(x_{20},\tau_n)$. In the same way, we can derive a contradiction. Therefore, the conclusion follows easily.

\textit{Step 3:} Reaching a contradiction.
Choose $\epsilon_0>0$ small, $R_{10}<0$ large negative and $R_{20}>0$ large positive such that
$$\phi(x-R_*)\geq1-\epsilon_0 \mbox{ for } x\leq R_{10}, \ \psi(x-R_*)\geq1-\epsilon_0 \mbox{ for } x\geq R_{20}$$
and
$$f'(u)<0 \mbox{ for } u\in[1-2\epsilon_0,1+2\epsilon_0], \ g'(v)<0 \mbox{ for } v\in[1-\frac{1}{2}\epsilon_0,1+2\epsilon_0].$$
Then choose $\epsilon\in(0,\epsilon_0)$ small such that
$$\wt V(x,t)\geq\psi(x-R_*+\epsilon),$$
$$ \phi(R_{10}-R_*+\epsilon)\geq\phi(R_{10}-R_*)-M_u(R_{10}),\ \psi(R_{20}-R_*+\epsilon)\leq\psi(R_{20}-R_*)+M_v(R_{20}),$$
$$\phi(x-R_*+\epsilon)\geq1-2\epsilon_0 \mbox{ for } x\leq R_{10},\ \psi(x-R_*+\frac{\epsilon}{2})\geq1-\frac{1}{2}\epsilon_0 \mbox{ for } x\geq R_{20}.$$

We consider the auxiliary problems
\begin{equation}\label{auxi-1}\left\{
\begin{array}{l}
\bar{U}_t-d_1\bar{U}_{xx}-c\bar{U}_x=f(\bar{U}), \quad \quad x<R_{10},\ t>0,\\
\bar{U}(R_{10},t)=\phi(R_{10}-R_*+\epsilon), \quad t>0,\\
\bar{U}(x,0)=1, \hskip 32mm x<R_{10}
\end{array}\right.
\end{equation}
and
\begin{equation}\label{auxi-2}\left\{
\begin{array}{l}
\bar{V}_t-d_2\bar{V}_{xx}-c\bar{V}_x=g(\bar{V}), \quad \quad x>R_{20},\ t>0,\\
\bar{V}(R_{20},t)=\psi(R_{20}-R_*+\epsilon), \ t>0,\\
\bar{V}(x,0)=\psi(x-R_*+\epsilon), \quad \quad x>R_{20}.
\end{array}\right.
\end{equation}
Since $1$ is the an upper solution of \eqref{auxi-1} and \eqref{auxi-2}, the unique solution of \eqref{auxi-1} and \eqref{auxi-2} are decreasing in $t$. Obviously, $\underline{U}(x,t)=\phi(x-R_*+\epsilon),\underline{V}(x,t)=\psi(x-R_*+\frac{\epsilon}{2})$ are the lower solution of \eqref{auxi-1} and \eqref{auxi-2}. By comparison principle, we have
$$\phi(x-R_*+\epsilon)\leq\bar{U}(x,t)\leq1 \mbox{ for all }t>0, \ x<R_{10},$$
$$\psi(x-R_*+\frac{\epsilon}{2})\leq\bar{V}(x,t)\leq1 \mbox{ for all }t>0,\ x>R_{20}.$$
Therefore,
$$\mathcal{U}(x):=\lim_{t\rightarrow\infty}\bar{U}(x,t)\geq\phi(x-R_*+\epsilon) \mbox{ for all } x<R_{10},$$
$$\mathcal{V}(x):=\lim_{t\rightarrow\infty}\bar{V}(x,t)\geq\psi(x-R_*+\frac{\epsilon}{2}) \mbox{ for all } x>R_{20}.$$
Moreover, $(\mathcal{U}(x),\mathcal{V}(x))$ satisfies
\begin{equation}\label{stationsolution1}
\begin{array}{l}
-d_1\mathcal{U}_{xx}-c\mathcal{U}_x=f(\mathcal{U}) \mbox{ in } (-\infty,R_{10}), \ \mathcal{U}(-\infty)=1, \ \mathcal{U}(R_{10})=\phi(R_{10}-R_*+\epsilon),\\
-d_2\mathcal{V}_{xx}-c\mathcal{V}_x=g(\mathcal{V}) \mbox{ in } (R_{20},\infty), \ \mathcal{V}(\infty)=1,\ \mathcal{V}(R_{20})=\phi(R_{20}-R_*+\frac{\epsilon}{2}).\\
\end{array}
\end{equation}
Write $\Phi(x)=\phi(x-R_*+\epsilon),\Psi(x)=\psi(x-R_*+\frac{\epsilon}{2})$. We notice that $(\Phi(x),\Psi(x))$ also satisfies \eqref{stationsolution1}.
Moveover,
$$1-2\epsilon_0\leq\Phi(x)\leq \mathcal{U}(x)\leq1 \mbox{ for } x\leq R_{10},$$
$$1-\frac{1}{2}\epsilon_0\leq\Psi(x)\leq \mathcal{V}(x)\leq1 \mbox{ for } x\geq R_{20}.$$
Hence $W_1(x)=\mathcal{U}(x)-\Phi(x)\geq0,\ W_2(x)=\mathcal{V}(x)-\Psi(x)\geq0$ and there exist $c_1(x),c_2(x)<0$ such that
$$f(\mathcal{U}(x))-f(\Phi(x))=c_1(x)W_1(x) \mbox{ in }(-\infty,R_{10}],$$
$$g(\mathcal{V}(x))-g(\Psi(x))=c_2(x)W_2(x) \mbox{ in }[R_{20},\infty).$$
Therefore
$$-d_1W_1''-cW_1'=c_1(x)W_1\mbox{ in }(-\infty,R_{10}],\ W_1(R_{10})=0,$$
$$-d_2W_2''-cW_2'=c_2(x)W_2\mbox{ in }[R_{20},\infty),\ W_2(R_{20})=0,$$
and by the maximum principle we deduce, for any $R_1<R_{10},R_2>R_{20},$
$$W_1(x)\leq W_1(R_1) \mbox{ in }[R_1,R_{10}],\ W_2(x)\leq W_2(R_2)\mbox{ in }[R_{20},R_2].$$
Letting $R_1\rightarrow-\infty$ and $R_2\rightarrow\infty$ we deduce that $W_1(x)\leq0 \mbox{ in }(-\infty,R_{10}]$ and $W_2(x)\leq0 \mbox{ in }[R_{20},\infty).$ It follows that $W_1\equiv0,W_2\equiv0$. Hence
$$\mathcal{U}(x)\equiv\Phi(x)=\phi(x-R_*+\epsilon),\mathcal{V}(x)\equiv\Psi(x)=\psi(x-R_*+\frac{\epsilon}{2}).$$
We now look at $(\wt U(x,t),\wt V(x,t))$, since $\wt U(x,t)$ satisfies the first equation in \eqref{auxi-1}, and for any $t\in \R^1$,
$$\wt U(x,t)\leq1,\ \wt U(R_{10},t)\leq \phi(R_{10}-R_*)-M_{R_{10}}\leq \phi(R_{10}-R_*+\epsilon).$$
In the same way, we know
$$\wt V(x,t)\geq \psi(x-R_*+\epsilon),\ \wt V(R_{20},t)\geq\psi(R_{20}-R_*)+M_{R_{20}}\geq\psi(R_{20}-R+\epsilon).$$
Consequently, we use the comparison principle to deduce that
$$\wt U(x,t+s)\leq\bar{U}(x,t) \mbox{ for all }t>0, x<R_{10},s\in \R^1,$$
$$\wt V(x,t+s)\geq\bar{V}(x,t) \mbox{ for all }t>0, x>R_{20},s\in \R^1.$$
Or equivalently,
$$\wt U(x,t)\leq\bar{U}(x,t-s) \mbox{ for all }t>s, x<R_{10},s\in \R^1,$$
$$\wt V(x,t)\geq\bar{V}(x,t-s) \mbox{ for all }t>s, x>R_{20},s\in \R^1.$$
Letting $s\rightarrow-\infty$ we obtain
\begin{equation}\label{lemequationinf}
\begin{array}{l}
\wt U(x,t)\leq \mathcal{U}(x)=\phi(x-R_*+\epsilon)\mbox{ for all } x<R_{10},t\in \R^1,\\
\wt V(x,t)\geq \mathcal{V}(x)=\psi(x-R_*+\frac{\epsilon}{2})\mbox{ for all } x>R_{20},t\in \R^1.
\end{array}
\end{equation}
By Step 2 and the continuity of $M_u(x),M_v(x)$ in $x$, we have
$$M_u(x)\geq\sigma>0 \mbox{ for }x\in[R_{10},R_*-\delta],$$
$$M_v(x)\geq\sigma>0 \mbox{ for }x\in[R_*,R_{20}].$$
If we choose $\epsilon_1\in (0,\epsilon]$ is small enough we have
$$\phi(x-R_*+\epsilon_1)\geq\phi(x-R_*)-\sigma \mbox{ for } x\in [R_{10},R_*-\delta],$$
$$\psi(x-R_*+\frac{\epsilon_1}{2})\leq\psi(x-R_*)+\sigma \mbox{ for } x\in [R_*,R_{20}],$$
and so
$$\wt U(x,t)-\phi(x-R_*+\epsilon_1)\leq\sigma-M_u(x)\leq0 \mbox{ for } x\in [R_{10},R_*-\delta], t\in\R^1,$$
$$\wt V(x,t)-\psi(x-R_*+\frac{\epsilon_1}{2})\geq M_v(x)-\sigma\geq0 \mbox{ for } x\in [R_*,R_{20}],t\in \R^1.$$
Therefore we can combine with \eqref{lemequationinf} to obtain
$$\wt U(x,t)-\phi(x-R_*+\epsilon_1)\leq0 \mbox { for } x\in (-\infty,R_*-\delta], t\in\R^1,$$
$$\wt V(x,t)-\psi(x-R_*+\frac{\epsilon_1}{2})\geq0 \mbox{ for } x\in [R_*,\infty),t\in \R^1,$$
for all small $\epsilon_1\in(0,\epsilon)$, which contradicts the definition of $R_u,R_v$. The proof is complete.
\end{proof}

\begin{lem}\label{lem:supst}
There exists a sequence $\{s_n\}\subset\mathbb{R}^1$ such that
\begin{equation*}
\begin{array}{l}
G(t+s_n)\rightarrow R_* \mbox{ as } n\rightarrow\infty  \mbox{ uniformly for } t \mbox { in compact subset of } \R^1,\\
\wt U(x,t+s_n)\rightarrow\phi(x-R_*) \mbox{ as } n\rightarrow\infty  \mbox{ uniformly for } (x,t) \mbox { in compact subset of } (-\infty,R_*]\times\mathbb{R}^1, \\
\wt V(x,t+s_n)\rightarrow\psi(x-R_*) \mbox{ as } n\rightarrow\infty \mbox{ uniformly for } (x,t)  \mbox { in compact subset of } [R_*,\infty)\times\mathbb{R}^1.
\end{array}
\end{equation*}
\end{lem}
\begin{proof}
There are two possibilities:

(i) \ $R_*=\sup_{t\in\R^1}G(t)$ is achieved at some finite time $t=s_0$,

(ii) $R_*>G(t) \mbox{ for all } t\in\R^1$ and $G(s_n)\rightarrow R_*$ along some unbounded sequence $s_n$.

In case (i), we have $G'(s_0)=0$, since $\wt U(x,t)\leq \phi(x-R_*)\mbox{ for } x\leq G(t), \ \wt V(x,t)\geq\psi(x-R_*) \mbox{ for } x\geq G(t) \mbox{ and } t\in \R^1$, with $\wt U(G(s_0),s_0)=\phi(G(s_0)-R_*)=\phi(0)=0, \ \wt V(G(s_0),s_0)=\psi(G(s_0)-R_*)=\psi(0)=0$, by strong maximum principle and Hopf boundary lemma, we have
$$\wt U_x(G(s_0),s_0)>\phi'(0) \mbox{ unless } \wt U(x,t)\equiv\phi(x-R_*) \mbox{ in } D_{0u}=\{(x,t):x\leq G(t),t\leq s_0 \},$$
$$\wt V_x(G(s_0),s_0)>\psi'(0) \mbox{ unless } \wt V(x,t)\equiv\psi(x-R_*) \mbox{ in } D_{0v}=\{(x,t):x\geq G(t),t\leq s_0 \}.$$
On the other hand, we know
$$-\mu_1\wt U_x(G(s_0),s_0)-\mu_2 \wt V_x(G(s_0),s_0)-c=G'(s_0)=0,$$
but in view of
$$-\mu_1\wt U_x(G(s_0),s_0)<-\mu_1\phi'(0),\ -\mu_2 \wt V_x(G(s_0),s_0)<-\mu_2\psi'(0),$$
we have
$$-\mu_1\wt U_x(G(s_0),s_0)-\mu_2 \wt V_x(G(s_0),s_0)-c<0.$$
Therefore $G(t)\equiv R_*,\wt U(x,t)\equiv \phi(x-R_*) \mbox{ in } D_{0u} \mbox{ and }\wt V(x,t)\equiv \psi(x-R_*) \mbox{ in } D_{0v}.$ Using the uniqueness of \eqref{lemequation1} with a given initial value, we conclude that
$$\wt U(x,t)\equiv \phi(x-R_*) \mbox{ for all }x\leq G(t), t\in\R^1,$$
$$\wt V(x,t)\equiv \psi(x-R_*) \mbox{ for all }x\geq G(t), t\in\R^1,$$
so we choose $s_n\equiv s_0$ and the conclusion of the lemma holds.
In case (ii), we consider the following sequence
$$U_n(x,t)=\wt U(x,t+l_n),\ V_n(x,t)=\wt V(x,t+l_n),\ G_n(t)=G(t+l_n).$$
Using the same method as in the proof of Lemma \ref{lem:sequencebounded}, we can choose a subsequence such that
$$U_n\rightarrow  \mathbf{U} \mbox{ in } C_{loc}^{1+\alpha,\frac{1+\alpha}{2}}(D_u),\ V_n\rightarrow  \mathbf{V} \mbox{ in } C_{loc}^{1+\alpha,\frac{1+\alpha}{2}}(D_v),\ G_n\rightarrow  \mathbf{G} \mbox{ in } C_{loc}^1(\R^1),$$
and $(\mathbf{U},\mathbf{V},\mathbf{G})$ satisfies \eqref{lemequation1}. Moreover,
$$\mathbf{G}(t)\leq R_*,\ \mathbf{G}(0)=R_*.$$
Hence we are back to case (i) and thus $\mathbf{U}(x,t)\equiv \phi(x-R_*) \mbox { in } D_u,\ \mathbf{V}(x,t)\equiv \psi(x-R_*) \mbox { in } D_v,$ and $\mathbf{G}\equiv R_*.$ The conclusion of the lemma follows easily.
\end{proof}
\begin{lem}\label{thm:sequence2}
There exists $T_k\rightarrow\infty$ such that
\begin{equation*}
\begin{array}{l}
\wt s(t+T_k)\rightarrow R_* \mbox{ as } k\rightarrow\infty  \mbox{ uniformly for } t \mbox { in compact subset of } \R^1,\\
\wt u(x,t+T_k)\rightarrow\phi(x-R_*) \mbox{ as } k\rightarrow\infty  \mbox{ uniformly for } (x,t) \mbox { in compact subset of } (-\infty,R_*]\times\R^1, \\
\wt v(x,t+T_k)\rightarrow\psi(x-R_*) \mbox{ as } k\rightarrow\infty \mbox{ uniformly for } (x,t)  \mbox { in compact subset of } [R_*,\infty)\times\R^1.
\end{array}
\end{equation*}
\end{lem}
\begin{proof}
For $k\geq1$, define
$$D^g=[-k,k], \ D^u=[-k,R_*],\ D^v=[R_*,k].$$
By Lemma \ref{lem:supst}, there exists $s_{n_k}$ such that
$$|G(t+s_{n_k})-R_*|\leq\frac{1}{k}\mbox{ for } t\in D^g, $$
$$|\wt U(x,t+s_{n_k})-\phi(x-R_*)|\leq \frac{1}{k} \mbox{ for }(x,t)\in D^u\times D^g,$$
$$|\wt V(x,t+s_{n_k})-\psi(x-R_*)|\leq \frac{1}{k} \mbox{ for }(x,t)\in D^v\times D^g.$$
By Lemma \ref{lem:sequencebounded}, there exists $t_{m_k}$ such that $t_{m_k}+s_{n_k}>k$ and
$$|\wt s(t+t_{m_k}+s_{n_k})-G(t+s_{n_k})|<\frac{1}{k} \mbox{ for } t\in D^g,$$
$$|\wt u(x,t+t_{m_k}+s_{n_k})-\wt U(x,t+s_{n_k})|\leq \frac{1}{k} \mbox{ for }(x,t)\in D^u\times D^g,$$
$$|\wt v(x,t+t_{m_k}+s_{n_k})-\wt V(x,t+s_{n_k})|\leq \frac{1}{k} \mbox{ for }(x,t)\in D^v\times D^g.$$
Therefore we can take $T_k=t_{m_k}+s_{n_k}$, then as $k\rightarrow\infty,T_k\rightarrow\infty$ and
$$|\wt s(t+T_k)-R_*|\leq\frac{2}{k}\rightarrow0 \mbox{ uniformly for } t \mbox { in compact subset of } \R^1,$$
$$|\wt u(x,t+T_k)-\phi(x-R_*)|\leq\frac{2}{k}\rightarrow0 \mbox{ uniformly for } (x,t) \mbox { in compact subset of } (-\infty,R_*]\times\R^1,$$
$$|\wt v(x,t+T_k)-\psi(x-R_*)|\leq\frac{2}{k}\rightarrow0 \mbox{ uniformly for } (x,t) \mbox { in compact subset of } [R_*,\infty)\times\R^1.$$
\end{proof}

\subsection{The Proof of Theorem \ref{thm:asymptoticbehavior}} We now apply Lemma \ref{lem:closeaftert}, which indicate that once $\wt u$ is close to $\phi(x-R_*)$ and $\wt v$ is close to $\psi(x-R_*)$ for some $T_k$, it remains close for all later time. Thus
$$\wt u(x,t)\rightarrow\phi(x-R_*), \ \wt v(x,t)\rightarrow\psi(x-R_*) \mbox{ as }t\rightarrow\infty.$$
Denote $x^*=R_*-2C$, we have
$$P(x,t)\rightarrow\phi(x-ct-x^*), \  Q(x,t)\rightarrow\psi(x-ct-x^*) \mbox{ as }t\rightarrow\infty.$$
Moreover, $|s(t)-ct-x^*|\rightarrow0 \mbox{ as } t\rightarrow\infty$.

Next, we prove $s'(t)\rightarrow c$.

If $s'(t)$ does not converge to $c$, there exist a $\epsilon>0$ and a sequence $\{t_k\}_{k=1}^{\infty}$ such that
$|s'(t_k)-c|>\epsilon \mbox{ for } k \mbox{ and }t_k\rightarrow\infty \mbox{ as }k\rightarrow\infty$. With no loss of generality, we assume $s'(t_k)>c+\epsilon$. Since $|s(t)-ct-x^*|\rightarrow0 \mbox{ as } t\rightarrow\infty$, we can make estimate about $\wt u \mbox{ and } \wt v$ at $x=x^*$
$$
\aligned
\wt u(x^*,t_k)=&\wt u(s(t_k)-ct_k,t_k)+O(s(t_k)-ct_k-x^*) \\
=&\wt u(s(t_k)-ct_k+2C,t_k)-\wt u_x(s(t_k)-ct_k+2C,t_k)(2C)+o((2C)^2)+O(s(t_k)-ct_k-x^*)\\
=&-\wt u_x(s(t_k)-ct_k+2C,t_k)(2C)+o((2C)^2)+O(s(t_k)-ct_k-x^*)
\endaligned
$$
and
$$
\aligned
\wt v(x^*,t_k)=-\wt v_x(s(t_k)-ct_k+2C,t_k)(2C)+o((2C)^2)+O(s(t_k)-ct_k-x^*).
\endaligned
$$
Using the same method, we have
$$\aligned
\phi(-2C)=\phi(0)-\phi'(0)(2C)+o((2C)^2)
\endaligned
$$
and
$$\aligned
\psi(-2C)=\psi(0)-\psi'(0)(2C)+o((2C)^2).
\endaligned
$$
Obviously, $O(s(t_k)-ct_k-x^*)\rightarrow0 \mbox{ as } k\rightarrow\infty$.
$$\aligned
&\liminf_{k\rightarrow\infty}(\mu_1(\wt u(x^*,t_k)-\phi(-2C))+\mu_2(\wt v(x^*,t_k)-\psi(-2C)))\\
>&(c+\epsilon)(2C)+o((2C)^2)-2cC-o((2C)^2)\\
=&2\epsilon C+o((2C)^2)\\
>&\epsilon C>0
\endaligned
$$
However, this is a contradiction with the asymptotic behavior of $\wt u\mbox{ and }\wt v$. Thus $s'(t)\rightarrow c$.
This completes the proof of Theorem \ref{thm:asymptoticbehavior}.\qed
\medskip

{\bf Acknowledgement}. The authors would like to thank Professor C. C. Chen for sending
the paper \cite{CC} to them, thank Professors M. Nagayama,
K. I. Nakamura, Y. Yamada and M. L. Zhou for helpful discussions.

\begin{center}
APPENDIX:The Proof of Theorem \ref{thm:existence1}
\end{center}
\begin{thm}\label{thm:localexistenceanduniqueness}
Under the assumption of Theorem \ref{thm:existence1}, there exists a unique solution to problem \eqref{p} for small time interval $0<t\leq T(T>0)$.
\end{thm}
\begin{proof}
Let \\$$y=x-s(t), t=t, P(x,t)=P(y+s(t),t)=u(y,t), Q(x,t)=Q(y+s(t),t)=v(y,t),$$
$$P_t=u_t-s'(t)u_y,Q_t=v_t-s'(t)v_y,P_x=u_y,Q_x=v_y,P_{xx}=u_{yy},Q_{xx}=v_{yy}.$$
Problem \eqref{p} can be transformed as follows:
\begin{equation}\label{p_1}
\left\{
\begin{array}{l}
u_{t}-d_1u_{yy}-s'(t)u_y=f(u),\hskip14mm y<0,\ t>0,\\
v_{t}-d_2v_{yy}-s'(t)v_y=g(v),\hskip15mm y>0,\ t>0,\\
u(0,t)=v(0,t)=0, \hskip26mm  \  t>0,\\
s'(t)=-\mu_1 u_{y}(0,t)-\mu_2 v_{y}(0,t),\quad \ \ \ t>0,\\
s(0)=s_0,\ s_0\in (-\infty,\infty),\\
u(y,0)=P_{0}(y+s_0)(y<0),\ v(y,0)=Q_{0}(y+s_0)(y>0).
\end{array}
\right.
\end{equation}
Assume
\begin{equation}\label{assume}
P_0\in C^2((-\infty,s_0]), Q_0\in C^2([s_0,\infty)),P_0(s_0)=Q_0(s_0)=0.
\end{equation}
Define
\begin{equation*}
\left.
\begin{array}{c}
D_{1T}=(-\infty,0)\times(0,T),D_{2T}=(0,\infty)\times(0,T),
\overline{D}_{1T}=(-\infty,0]\times[0,T],\overline{D}_{2T}=[0,\infty)\times[0,T],\\
D_1=\{u(y,t)\in C(\overline{D}_{1T}),u(y,0)=P_0(y+s_0)\},
\widetilde{D}_{1T}=\{u\in D_1,\displaystyle\sup_{y\leq0,t\in[0,T]}|u(y,t)-u(y,0)|\leq1\},\\
D_2=\{v(y,t)\in C(\overline{D}_{2T}),v(y,0)=Q_0(y+s_0)\},
\widetilde{D}_{2T}=\{v\in D_2,\displaystyle\sup_{y\geq0,t\in[0,T]}|v(y,t)-v(y,0)|\leq1\},\\
D_3=\{s\in C^1[0,T],s(0)=s_0,s'(0)=-\mu_1P'_0(s_0)-\mu_2Q'_0(s_0)\},\\
\widetilde{D}_{3T}=\{s\in D_3,\displaystyle\sup_{t\in[0,T]}|s'(t)-s'(0)|\leq1\},
\end{array}
\right.
\end{equation*}
then $D=\widetilde{D}_{1T}\times\widetilde{D}_{2T}\times\widetilde{D}_{3T}$ is a completed metric space with the distance defined as follows:
$$d((u_1,v_1,s_1),(u_2,v_2,s_2))=\|u_1-u_2\|_{C(\widetilde{D}_{1T})}+\|v_1-v_2\|_{C(\widetilde{D}_{2T})}+\|s'_1-s'_2\|_{C([0,T])},$$
for $s_1,s_2\in\widetilde{D}_{3T}, s_1(0)=s_2(0)=s_0$, we have $\|s_1-s_2\|_{C([0,T])}\leq T\|s'_1-s'_2\|_{C([0,T])}.$

Next we will prove the existence and uniqueness result by using the contraction mapping theorem. Applying $L^p$ theory and Sobolev embedding theorem, for any $(u,v,s)\in D$, the following problem
\begin{equation*}
\left\{
\begin{array}{l}
\overline{u}_{t}-d_1\overline{u}_{yy}-s'(t)\overline{u}_y=f(u),\hskip14mm y<0,\ t>0,\\
\overline{u}(0,t)=0, \hskip42mm  \ t>0,\\
\overline{u}(y,0)=P_{0}(y+s_0),\hskip27mm y\leq0,
\end{array}
\right.
\end{equation*}
has a unique solution $\overline{u}\in C^{1+\alpha,\frac{1+\alpha}{2}}(\overline{D}_{1T})$.

In the same way, the following problem
\begin{equation*}
\left\{
\begin{array}{l}
\overline{v}_{t}-d_2\overline{v}_{yy}-s'(t)\overline{v}_y=g(v),\hskip15mm y>0,\ t>0,\\
\overline{v}(0,t)=0, \hskip42mm  \ t>0,\\
\overline{v}(y,0)=Q_{0}(y+s_0),\hskip27mm y\geq0,
\end{array}
\right.
\end{equation*}
has a unique solution $\overline{v}\in C^{1+\alpha,\frac{1+\alpha}{2}}(\overline{D}_{2T})$. Moreover, we have
$$\|\overline{u}\|_{C^{1+\alpha,\frac{1+\alpha}{2}}(\overline{D}_{1T})}+\|\overline{v}\|_{C^{1+\alpha,\frac{1+\alpha}{2}}(\overline{D}_{2T})}\leq C_1,$$
where $C_1$ depends on $\alpha, |s'(0)|,\|P_0\|_{C^{1+\alpha}((-\infty,0])}, \|Q_0\|_{C^{1+\alpha}([0,\infty))}$.
Define new free boundary
$$\overline{s}(t)=-\mu_1\int_0^t\overline{u}_y(0,\tau)d\tau-\mu_2\int_0^t\overline{v}_y(0,\tau)d\tau,$$
so, $\overline{s}'(t)=-\mu_1 \overline{u}_{y}(0,t)-\mu_2 \overline{v}_{y}(0,t), \overline{s}(0)=0, \overline{s}'\in C^{\frac{\alpha}{2}}([0,T])$ and $\|\overline{s}'\|_{C^{\frac{\alpha}{2}}([0,T])}\leq C_2,$ where $C_2$ depends on $\alpha, |s'(0)|,\|P_0\|_{C^{1+\alpha}((-\infty,0])}, \|Q_0\|_{C^{1+\alpha}([0,\infty))}$.

Defining a mapping
$\mathcal{F}:D\rightarrow C(\overline{D}_{2T})\times C(\overline{D}_{2T})\times C^1([0,T]),$ by
$$\mathcal{F}(u,v,s)=(\overline{u},\overline{v},\overline{s}).$$
Obviously, $(u,v,s)$ is the fixed point of $\mathcal{F}$ if and only if $(u,v,s)$ is the solution of \eqref{p_1}.
Noting that
$$\aligned|\overline{u}(y,t)-u(y,0)|=&t^{\frac{1+\alpha}{2}}\frac{|\overline{u}(y,t)-u(y,0)|}{t^{\frac{1+\alpha}{2}}}\\
\leq&T^{\frac{1+\alpha}{2}}\displaystyle\sup_{y\leq0,t\in[0,T]}\frac{|\overline{u}(y,t)-u(y,0)|}{t^{\frac{1+\alpha}{2}}}\\
\leq&T^{\frac{1+\alpha}{2}}\|\overline{u}\|_{C^{1+\alpha,\frac{1+\alpha}{2}}(\overline{D}_{1T})}.\\
\endaligned$$
Therefore, for any $y\leq0,t\in[0,T],$ we have
$$\displaystyle\sup_{y\leq0,t\in[0,T]}|\overline{u}(y,t)-u(y,0)|\leq C_1T^{\frac{1+\alpha}{2}}.$$
In the same way, we know that
$$\displaystyle\sup_{y\geq0,t\in[0,T]}|\overline{v}(y,t)-v(y,0)|\leq C_1T^{\frac{1+\alpha}{2}}.$$
Moreover,
$$\overline{s}'(t)-s'(0)=-\mu_1\overline{u}_y(0,t)-\mu_2\overline{v}_y(0,t)+\mu_1u_y(0,0)+\mu_2v_y(0,0),$$
Consequently,
$$\aligned\displaystyle\sup_{t\in[0,T]}|\overline{s}'(t)-s'(0)|
\leq&\mu_1\sup_{t\in[0,T]}|\overline{u}_y(0,t)-u_y(0,0)|+\mu_2\sup_{t\in[0,T]}|\overline{v}_y(0,t)-v_y(0,0)|\\
\leq&\mu_1T^{\frac{\alpha}{2}}\|\overline{u}\|_{C^{1+\alpha,\frac{1+\alpha}{2}}(\overline{D}_{1T})}
+\mu_2T^{\frac{\alpha}{2}}\|\overline{v}\|_{C^{1+\alpha,\frac{1+\alpha}{2}}(\overline{D}_{2T})}\\
\leq&(\mu_1+\mu_2)C_1T^{\frac{\alpha}{2}}.
\endaligned
$$
Choose $T\leq\min\{((\mu_1+\mu_2)C_1)^{-\frac{2}{\alpha}},C_1^{-\frac{2}{1+\alpha}}\},$ then the mapping $\mathcal{F}$ maps $D$ into itself.

Next, we prove $\mathcal{F}$ is a contraction mapping if $T$ is sufficiently small. For $(u_i,v_i,s_i)\in D$, denote $(\overline{u}_i,\overline{v}_i,\overline{s}_i)=\mathcal{F}(u_i,v_i,s_i)$, then $\|\overline{u}_i\|_{C^{1+\alpha,\frac{1+\alpha}{2}}(\overline{D}_{1T})}\leq C_1, \|\overline{v}_i\|_{C^{1+\alpha,\frac{1+\alpha}{2}}(\overline{D}_{2T})}\leq C_1, \|\overline{s}_i'\|_{C^{\frac{\alpha}{2}}([0,T])}\leq C_2.$ denote
$U=\overline{u}_1-\overline{u}_2$, then
\begin{equation*}
\left\{
\begin{array}{l}
U_{t}-d_1U_{yy}-s'_1(t)U_y=f(u_1)-f(u_2)+(s'_1(t)-s'_2(t))\overline{u}_{2y},\hskip15mm y>0,\ t>0,\\
U(0,t)=0, \hskip93mm  \ t>0,\\
U(y,0)=0,\hskip93mm y\geq0.
\end{array}
\right.
\end{equation*}
Applying $L^p$ estimate and Sobolev embedding theorem
$$\|\overline{u}_1-\overline{u}_2\|_{C^{1+\alpha,\frac{1+\alpha}{2}}(\overline{D}_{1T})}\leq C_3\left(\|u_1-u_2\|_{C(\overline{D}_{1T})}+\|s_1-s_2\|_{C^1([0,T])}\right),$$
where $C_3$ depends on $\alpha, |s'(0)|,\|P_0\|_{C^{1+\alpha}((-\infty,0])}\mbox{ and }\|Q_0\|_{C^{1+\alpha}([0,\infty))}$.

In the same way, let $V=\overline{v}_1-\overline{v}_2$, then
$$\|\overline{v}_1-\overline{v}_2\|_{C^{1+\alpha,\frac{1+\alpha}{2}}(\overline{D}_{2T})}\leq C_4\left(\|v_1-v_2\|_{C(\overline{D}_{2T})}+\|s_1-s_2\|_{C^1([0,T])}\right),$$
where $C_4$ depends on $\alpha, |s'(0)|,\|P_0\|_{C^{1+\alpha}((-\infty,0])}\mbox{ and } \|Q_0\|_{C^{1+\alpha}([0,\infty))}$.
Assume $T\leq1$, we have $$\|\overline{s}'_1-\overline{s}'_2\|_{C^{\frac{\alpha}{2}}([0,T])}\leq\mu_1\|\overline{u}_{1y}-
\overline{u}_{2y}\|_{C^{\frac{\alpha}{2},0}(\overline{D}_{1T})}
+\mu_2\|\overline{v}_{1y}-\overline{v}_{2y}\|_{C^{\frac{\alpha}{2},0}(\overline{D}_{2T})},$$
$$\aligned&\|\overline{u}_1-\overline{u}_2\|_{C^{1+\alpha,\frac{1+\alpha}{2}}(\overline{D}_{1T})}+
\|\overline{v}_1-\overline{v}_2\|_{C^{1+\alpha,\frac{1+\alpha}{2}}(\overline{D}_{2T})}+
\|\overline{s}'_1-\overline{s}'_2\|_{C^{\frac{\alpha}{2}}([0,T])}\\
&\leq C_5\left(\|u_1-u_2\|_{C(\overline{D}_{1T})}+\|v_1-v_2\|_{C(\overline{D}_{2T})}+\|s'_1-s'_2\|_{C([0,T])}\right),
\endaligned$$
where $C_5$ depends on $\alpha, |s'(0)|,\|P_0\|_{C^{1+\alpha}((-\infty,0])}, \|Q_0\|_{C^{1+\alpha}([0,\infty))},\mu_1\mbox{ and }\mu_2$.

So, if we take $T:=\min\left\{1, ((\mu_1+\mu_2)C_1)^{-\frac{2}{\alpha}},C_1^{-\frac{2}{1+\alpha}},\left(\frac{1}{2C_5}\right)^{\frac{2}{\alpha}}\right\}$, then
$$\aligned&\|\overline{u}_1-\overline{u}_2\|_{C(\overline{D}_{1T})}+\|\overline{v}_1-\overline{v}_2\|_{C(\overline{D}_{2T})}
+\|\overline{s}'_1-\overline{s}'_2\|_{C([0,T])}\\
&\leq T^{\frac{1+\alpha}{2}}\left(\|\overline{u}_1-\overline{u}_2\|_{C^{1+\alpha,\frac{1+\alpha}{2}}(\overline{D}_{1T})}+
\|\overline{v}_1-\overline{v}_2\|_{C^{1+\alpha,\frac{1+\alpha}{2}}(\overline{D}_{2T})}\right)+T^{\frac{\alpha}{2}}\|\overline{s}'_1
-\overline{s}'_2\|_{C^{\frac{\alpha}{2}}([0,T])}
\\
&\leq C_5T^{\frac{\alpha}{2}}\left(\|u_1-u_2\|_{C(\overline{D}_{1T})}+\|v_1-v_2\|_{C(\overline{D}_{2T})}
+\|s'_1-s'_2\|_{C([0,T])}\right)\\
&\leq\frac{1}{2}\left(\|u_1-u_2\|_{C(\overline{D}_{1T})}+\|v_1-v_2\|_{C(\overline{D}_{2T})}
+\|s'_1-s'_2\|_{C([0,T])}\right).
\endaligned$$
For such $T$, $\mathcal{F}$ is a contraction mapping on $D$. By contraction mapping theorem, $\mathcal{F}$ has a unique fixed point $(u,v,s)\in D$.
\end{proof}
\begin{thm}\label{thm:globalexistenceanduniqueness}
Under the assumption of Theorem \ref{thm:existence1}, the unique solution of the problem \eqref{p} exists, and it can be extended to $[0,T_{max})$, where $T_{max}\leq\infty$.
\end{thm}
\begin{proof}
In order to prove the present theorem, we argue it indirectly. Assume that $T_{max}<\infty$. Since $|s'(t)|\leq H$ in $[0,T_{max})$, using bootstrap argument and Schauder's estimate yields a priori bound of $|P(x,t)|_{C^{1+\alpha}((-\infty,s(t)])}+|Q (x,t)|_{C^{1+\alpha}([s(t),\infty))}$ for all $t\in[0, T_{max})$. Let the bound be $C_6$. It follows from the proof of Theorem \ref{thm:localexistenceanduniqueness} that there exists a $\tau$ depending only on $T_{max}, C_1$ such that the solution of the problem \eqref{p} with the initial time $T_{max}-\frac{\tau}{2}$ can be extended uniquely to the time $T_{max}-\frac{\tau}{2}+\tau$ that contradicts the assumption. This completes the proof.
\end{proof}

\end{document}